\documentclass[12pt,oneside,a4paper,draft,titlepage]{scrartcl}		

\usepackage[left=2.5cm,right=2.5cm,top=2cm,bottom=3cm]{geometry}

\usepackage{tikz}
\usepackage{tikz-3dplot}
\usetikzlibrary{calc,matrix,patterns,intersections,shapes}
\usetikzlibrary{positioning}

\usepackage[latin1]{inputenc}
\usepackage{amsmath}
\usepackage{mathrsfs}
\usepackage{amsthm}
\usepackage{amsfonts}
\usepackage{empheq}
\usepackage{pifont}
\usepackage{enumerate}

\usepackage[draft=false]{hyperref}

\usepackage{stmaryrd}

\usepackage[english]{babel}
\usepackage{units}
\usepackage{dsfont}
\usepackage{graphicx}
\usepackage{array}

\usepackage{setspace}

\usepackage{color}

\usepackage{nicefrac}

\usepackage{fancybox}
\usepackage{float}
\usepackage{amssymb}

\usepackage{nicefrac}

\usepackage{listings}

\usepackage{color}

\usepackage{times}

\newcommand*{\TextWidth}{4.2em}
\newcommand*{\myheight}{1.7em}

\newcommand{\Z}{{\mathbb Z}}
\newcommand{\N}{{\mathbb N}}  
\newcommand{\R}{{\mathbb R}}
\newcommand{\pz}{{\mathbb P}}

\newtheorem{theorem}{Theorem}
\newtheorem*{theoremB}{Theorem B}
\newtheorem{proposition}{Proposition}
\newtheorem{lemma}{Lemma}
\newtheorem*{lemmaDW1}{Lemma DWa}
\newtheorem*{lemmaDW2}{Lemma DWb}

\theoremstyle{definition}

\begin{document}

\pagenumbering{arabic}										
\setcounter{page}{1}										

\begin{center}
{\LARGE\bfseries Complete Subgraphs of the Coprime Hypergraph of Integers III: Construction}\\
\vspace{5mm}
{\large Jan-Hendrik de Wiljes}
\end{center}
\vspace{5mm}
{\bfseries Abstract.} The coprime hypergraph of integers on $n$ vertices $CHI_k(n)$ is defined via vertex set $\{1,2,\dots,n\}$ and hyperedge set $\{\{v_1,v_2,\dots,v_{k+1}\}\subseteq\{1,2,\dots,n\}:\gcd(v_1,v_2,\dots,v_{k+1})=1\}$. In this article we present ideas on how to construct maximal subgraphs in $CHI_k(n)$. This continues the author's earlier work, which dealt with bounds on the size and structural properties of these subgraphs. We succeed in the cases $k\in\{1,2,3\}$ and give promising ideas for $k\geq 4$.\\
\vspace{1mm}
{\bfseries Keywords.} Hypergraphs on integers, clique, matching\\
\vspace{1mm}
{\bfseries 2010 Mathematics Subject Classification.} 11B75, 05C65, 05C69, 05C70

\section{Introduction}
In this third part of a series of three articles we continue the work in \cite{dewiljes:2017a,dewiljes:2017b}. References and notation can be found there. The object of interest is the uniform \textit{coprime hypergraph of integers} $CHI_k$, which has vertex set $\Z$ and a $(k+1)$-hyperedge exactly between every $k+1$ elements of $\Z$ which have (not necessarily pairwise) greatest common divisor equal to $1$. In particular we are interested in the subgraph $CHI_k(n)$ of $CHI_k$ which is induced by the vertex set $[n]:=\N\cap[1,n]$.

Our main focus lies on the vertex subsets of $CHI_k(n)$ which induce complete subgraphs:
\[
CS_k(n):=\{A\subseteq[n]:o_p(A)\leq k\text{ for all }p\in\pz\},
\]
where $o_p(A):=|A\cap p\N|$. We define $CS_k^{max}(n)$ to be the set of elements from $CS_k(n)$ which have maximal cardinality $\mathfrak{cn}_k(n)$ among the elements of $CS_k(n)$. Elements from $CS_k^{max}(n)$ are called \textit{maximum} $(n,k)$-\textit{shelves}.

In \cite{dewiljes:2017a} questions concerning the cardinality of maximum shelves were considered. The maximal number of prime divisors of elements in (maximum) shelves was investigated in \cite{dewiljes:2017b}. It was shown that there exists a maximum shelf containing only elements with at most two prime divisors. Following this idea, we will construct such maximum $(n,3)$-shelves in this paper. Well known results from matching theory will be applied in the process.

We should mention that the ``direct" approach of expanding a maximum $(n,k)$-shelf does in general not lead to a maximum $(m,k)$-shelf for $m>n$. For example the maximum $(6,2)$-shelf $\{1,2,3,5,6\}$ is no subset of any maximum $(m,2)$-shelf for $m\geq 9$, since the number $6$ is not contained in any of those shelves. Using results from \cite{dewiljes:2017b} the following can be shown (we denote by $\pi$ the \textit{prime counting function}, by $\omega(a)$ the \textit{number of prime divisors} and by $\mathrm{ssp}(a)$ the \textit{second smallest prime divisor} of $a$):

\begin{proposition}
Let $k>2$, $n$ with $\pi(\sqrt{n})-\pi(\sqrt[k]{n})\geq 2$ and
\[
m\geq\max_{A\in CS_k^{max}(n)}\min_{\genfrac{}{}{0pt}{}{a\in A}{\omega(a)>1}}\mathrm{ssp(a)}^k
\]
be positive integers. Then no maximum $(n,k)$-shelf is a subset of any maximum $(m,k)$-shelf.
\end{proposition}

\section{Shifting and a tool from matching theory}

Similarly as in \cite{dewiljes:2017b} we will consider a transformation which preserves the condition $o_p(A)\leq k$, thus, mapping shelves onto shelves. We will use a specific order $\prec$ on $\N_{>1}$ to make sure that this operation is \textit{idempotent} (which facilitates the formulation). For $a\neq b\in\N_{>1}$ we define (using $e_p(a)$ as the \textit{exponent} of $p$ in the prime factorization of $a$)
\begin{align*}
a\prec b:\Longleftrightarrow~&\omega(a)<\omega(b)\text{ or}\\
&(\omega(a)=\omega(b)\text{ and }e_q(a)<e_q(b)\text{ with }q=\min\{p\in\pz:e_p(a)\neq e_p(b)\}).
\end{align*}

For a $(n,k)$-shelf $A$ we carry out the following steps ($p_i$ denotes the $i$th prime number):
\begin{itemize}
\item Set $A_0:=A$,
\item for every $i\in[\pi(n)]$ (in ascending order) we pick $e_i:=\min\{o_{p_i}(A_{i-1}),\lfloor\log_{p_i}n\rfloor\}$ (with respect to $\prec$) smallest multiples $a_{i1},a_{i2},\dots,a_{ie_i}$ of $p_i$ from $A_{i-1}$ and define
\[
A_i:=(A_{i-1}\setminus\{a_{i1},a_{i2},\dots,a_{ie_i}\})\cup\{p_i,p_i^2,\dots,p_i^{e_i}\}.
\]
\end{itemize}
This procedure, which essentially substitutes elements of $A$ with powers of one of its prime divisors (not necessarily the smallest), defines the \textit{shifting operation}
\[
\mathcal{S}:CS_k(n)\rightarrow CS_k(n),\quad A\mapsto A_{\pi(n)}.
\]

\begin{lemma}
For every $A\in CS_k^{max}(n)$ and every $i\in[\pi(n)]$ we have $e_i=\lfloor\log_{p_i}n\rfloor$.
\end{lemma}

\begin{proof}
Otherwise $A\cup\{p_i^j\}$ would be a larger $(n,k)$-shelf than $A$ for some $j\leq\lfloor\log_{p_i}n\rfloor$.
\end{proof}

\begin{lemma}\label{le:shiftingPreservesSize}
We have $|S(A)|=|A|$ for $A\in CS_k(n)$ and especially $\mathcal{S}(CS_k^{max}(n))\subseteq CS_k^{max}(n)$.
\end{lemma}
\begin{proof}
We have $|\{a_{i1},a_{i2},\dots,a_{ie_i}\}|=|\{p_i,p_i^2,\dots,p_i^{e_i}\}|$ and from the definition of $\prec$ we get $A_{i-1}\cap\{p_i,p_i^2,\dots,p_i^{e_i}\}\subseteq\{a_{i1},a_{i2},\dots,a_{ie_i}\}$ for every $i\in[\pi(n)]$. This implies $|A_i|=|A_{i-1}|$.
\end{proof}

We use the notation $\pz(x,y):=\pz\cap(x,y]$ for $x,y\in\R$. Since $S(\pz(1,n))=\pz(1,n)\in CS_1(n)$ and $S(\pz(1,n)\cup\{p^2:p\in\pz(1,\sqrt{n})\})=\pz(1,n)\cup\{p^2:p\in\pz(1,\sqrt{n})\}\in CS_2(n)$ and both sets cannot be expanded, as proper supersets do not belong to $CS_1(n)$ resp. $CS_2(n)$, we obtain

\begin{proposition}
For $n\in\N$ we have $\pz(1,n)\in CS_1^{max}(n)$ and $\pz(1,n)\cup\{p^2:p\in\pz(1,\sqrt{n})\}\in CS_2^{max}(n)$.
\end{proposition}

Thus, the first interesting case is $k=3$. To approach this problem we are going to use

\begin{lemma}\label{le:StructureShiftedThree}
For every $A\in CS_3^{max}(n)$ the set $S(A)$ contains only the number $1$, prime powers and elements of the form $pq$ with $p\in\pz(\sqrt[3]{n},\sqrt{n})$ and $q\in\pz(\sqrt[3]{n},n)$.
\end{lemma}

In view of Lemmas \ref{le:shiftingPreservesSize} and \ref{le:StructureShiftedThree} we solely need to maximize the number of elements $pq$ with $p\in\pz(\sqrt[3]{n},\sqrt{n})$ and $q\in\pz(\sqrt{n},n)$, while using every $p$ at most once and every $q$ at most twice, to construct a maximum $(n,3)$-shelf. To find such a maximal number of products we will make use of the following graph theoretical result from Claude Berge \cite{berge:1957}:

\begin{theoremB}[Berge 1957]
A matching $M$ of a graph $G$ is a maximum matching if and only if there exists no $M$-augmenting path in $G$.
\end{theoremB}

\section{Construction procedure for $k=3$}

We build the bipartite graph $G(n)$ (see Figure \ref{fi:BipartiteGraphMaximumMatching}) in the following way:
\begin{itemize}
\item It has vertex multiset\footnote{We could also view it as a set of size $2\pi(n)-\pi(\sqrt{n})-\pi(\sqrt[3]{n})$ and label the vertices accordingly. Our way avoids having to distinguish between vertex and label, which could become confusing.} $\pz(\sqrt[3]{n},\sqrt{n})\cup\pz(\sqrt{n},n)\cup\pz(\sqrt{n},n)$ and
\item an edge exactly between every $p\in\pz(\sqrt[3]{n},\sqrt{n})$ and $q\in\pz(\sqrt{n},n)$ with $pq\leq n$.
\end{itemize}

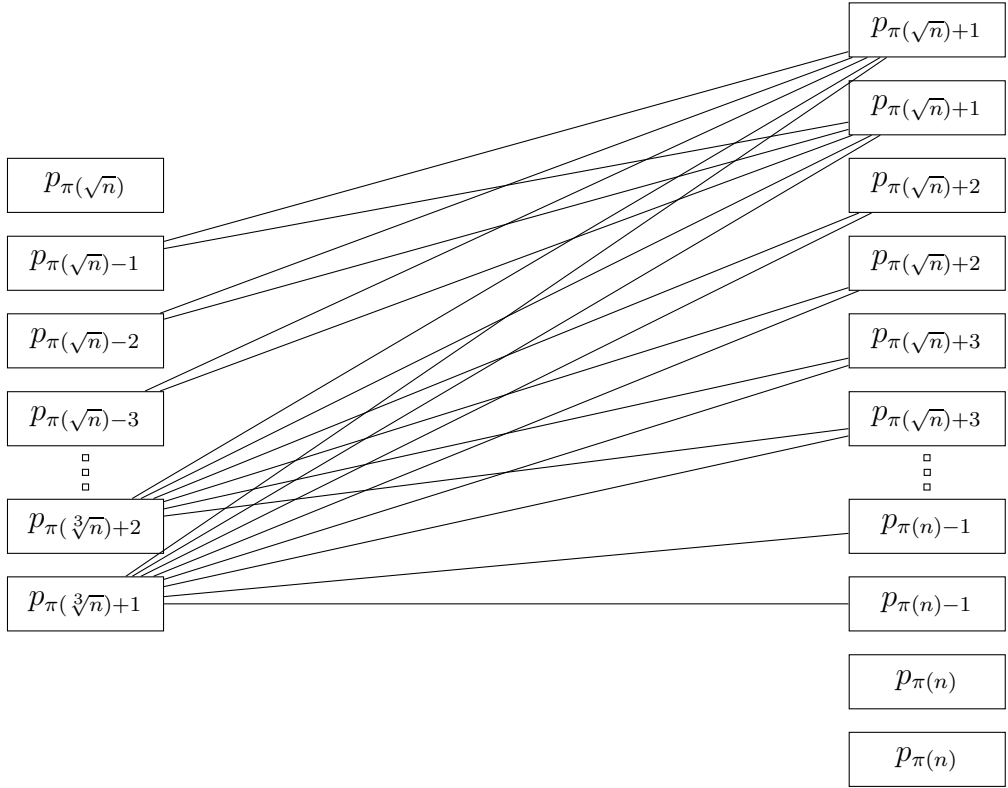
\begin{figure}[h]
\centering
\begin{tikzpicture}
\matrix[matrix, row sep=0.3cm, column sep=9cm]
{ & \node[text width=\TextWidth, draw, align=center, minimum height=\myheight] (q1) {$p_{\pi(\sqrt{n})+1}$}; \\
& \node[text width=\TextWidth, draw, align=center, minimum height=\myheight] (q2) {$p_{\pi(\sqrt{n})+1}$}; \\
\node[text width=\TextWidth, draw, align=center, minimum height=\myheight] (p1) {$p_{\pi(\sqrt{n})}$}; & \node[text width=\TextWidth, draw, align=center, minimum height=\myheight] (q3) {$p_{\pi(\sqrt{n})+2}$}; \\
\node[text width=\TextWidth, draw, align=center, minimum height=\myheight] (p2) {$p_{\pi(\sqrt{n})-1}$}; & \node[text width=\TextWidth, draw, align=center, minimum height=\myheight] (q4) {$p_{\pi(\sqrt{n})+2}$}; \\
\node[text width=\TextWidth, draw, align=center, minimum height=\myheight] (p3) {$p_{\pi(\sqrt{n})-2}$}; & \node[text width=\TextWidth, draw, align=center, minimum height=\myheight] (q5) {$p_{\pi(\sqrt{n})+3}$}; \\
\node[text width=\TextWidth, draw, align=center, minimum height=\myheight] (p4) {$p_{\pi(\sqrt{n})-3}$}; & \node[text width=\TextWidth, draw, align=center, minimum height=\myheight] (q6) {$p_{\pi(\sqrt{n})+3}$}; \\
\node[draw, inner sep=0.04cm] (h1) {}; & \node[draw,inner sep=0.04cm] (h2) {}; \\
\node[text width=\TextWidth, draw, align=center, minimum height=\myheight] (p5) {$p_{\pi(\sqrt[3]{n})+2}$}; & \node[text width=\TextWidth, draw, align=center, minimum height=\myheight] (q7) {$p_{\pi(n)-1}$}; \\
\node[text width=\TextWidth, draw, align=center, minimum height=\myheight] (p6) {$p_{\pi(\sqrt[3]{n})+1}$}; & \node[text width=\TextWidth, draw, align=center, minimum height=\myheight] (q8) {$p_{\pi(n)-1}$}; \\
& \node[text width=\TextWidth, draw, align=center, minimum height=\myheight] (q9) {$p_{\pi(n)}$}; \\
& \node[text width=\TextWidth, draw, align=center, minimum height=\myheight] (q10) {$p_{\pi(n)}$}; \\
};
\node[draw,inner sep=0.04cm,below=1mm of h1] {};
\node[draw,inner sep=0.04cm,above=1mm of h1] {};
\node[draw,inner sep=0.04cm,below=1mm of h2] {};
\node[draw,inner sep=0.04cm,above=1mm of h2] {};
\draw[-] (p2) edge (q1);
\draw[-] (p2) edge (q2);
\draw[-] (p3) edge (q1);
\draw[-] (p3) edge (q2);
\draw[-] (p4) edge (q1);
\draw[-] (p4) edge (q2);
\draw[-] (p5) edge (q1);
\draw[-] (p5) edge (q2);
\draw[-] (p5) edge (q3);
\draw[-] (p5) edge (q4);
\draw[-] (p5) edge (q5);
\draw[-] (p5) edge (q6);
\draw[-] (p6) edge (q1);
\draw[-] (p6) edge (q2);
\draw[-] (p6) edge (q3);
\draw[-] (p6) edge (q4);
\draw[-] (p6) edge (q5);
\draw[-] (p6) edge (q6);
\draw[-] (p6) edge (q7);
\draw[-] (p6) edge (q8);
\end{tikzpicture}
\caption{Graph $G(n)$ for the construction of maximum $(n,3)$-shelves.}
\label{fi:BipartiteGraphMaximumMatching}
\end{figure}

To construct a maximum matching in $G(n)$, we use the following procedure:

\begin{enumerate}
\item Start with the empty matching $M$ and draw $G(n)$ as shown in Figure \ref{fi:BipartiteGraphMaximumMatching}.
\item Take the highest (meaning the smallest) not matched element $q$ from $\pz(\sqrt{n},n)\cup\pz(\sqrt{n},n)$ and find the largest (also the highest) not matched element $p$ from $\pz(\sqrt[3]{n},\sqrt{n})$ being adjacent to $q$. Expand $M$ by $\{p,q\}$. 
\item Repeat the second step until it is not possible anymore.
\item Denote the resulting matching by $M(n)$ and the set of not matched elements from\\ $\pz(\sqrt[3]{n},\sqrt{n})$ by $F(n)$.
\end{enumerate}

\newpage

\begin{lemma}
For every $n\in\N$ the matching $M(n)$ is maximum in $G(n)$.
\end{lemma}

\begin{proof}
Suppose, $M(n)$ is not a maximum matching. From Theorem B we know that we find a shortest $M(n)$-augmenting path
\[
a_1,b_1,a_2,b_2,\dots,a_r,b_r,
\]
with $a_i\in\pz(\sqrt[3]{n},\sqrt{n})$, $b_i\in\pz(\sqrt{n},n)$ for $i\in[r]$. From the construction procedure follows that at least one of two adjacent vertices in $G(n)$ are matched. Therefore, we have $r\geq 2$. Since the construction method always picks the largest available element from $\pz(\sqrt[3]{n},\sqrt{n})$ and $\{a_2,b_1\}\in M(n)$ holds, we have $a_1<a_2$. From this we get
\[
a_1b_2<a_2b_2\leq n,
\]
meaning there is an edge between $a_1$ and $b_2$ in $G(n)$. But then
\[
a_1,b_2,a_3,b_3,\dots,a_r,b_r
\]
would be a shorter $M(n)$-augmenting path which yields a contradiction.
\end{proof}

Combining this result with Lemmas \ref{le:shiftingPreservesSize} and \ref{le:StructureShiftedThree} yields:

\begin{theorem}
For every $n\in\N$ the set 
\[
\{1\}\cup\left\{p^i\leq n:p\in\pz, i\in\{1,2,3\}\right\}\cup M(n)\cup\left\{a_1a_2,a_3a_4,\dots,a_{2\lfloor\nicefrac{s}{2}\rfloor-1}a_{2\lfloor\nicefrac{s}{2}\rfloor}\right\},
\]
where $F(n)=\{a_1,a_2,\dots,a_s\}$, is a maximum $(n,3)$-shelf.
\end{theorem}

By carefully examining the changes of $M(n)$ when it is transitioned into $M(n+1)$ one gets

\begin{proposition}
For $n\in\N$ we have
\[
\mathfrak{cn}_k(n+1)\in\begin{cases}
\{\mathfrak{cn}_k(n)+1\}&\text{if }n+1\in\pz,\\
\{\mathfrak{cn}_k(n),\mathfrak{cn}_k(n)+1\}&\text{if }\sqrt[3]{n+1}\in\pz\text{ or }n+1=pq\text{ for }p,q\in\pz\\
&\text{ with }\sqrt[3]{n}<q\leq\sqrt{n}<p,\\
\{\mathfrak{cn}_k(n),\mathfrak{cn}_k(n)+1,\mathfrak{cn}_k(n)+2\}&\text{if }\sqrt{n+1}\in\pz,\\
\{\mathfrak{cn}_k(n)\}&else.
\end{cases}
\]
\end{proposition}

The true value in the second and the third case depends on the number of neighbours of a given prime $p$ in $G(n)$ and on how the number of primes lying in $F(n)$. Since both quantities are difficult to determine exactly and also vary largely, it does not seem possible to get good approximations for $\mathfrak{cn}_k(n)$ from the presented method.

\section{Construction ideas for $k\geq 4$}

Adjusting the previous method will in general not yield a maximum $(n,k)$-shelf for $k\geq 4$.

On the one hand this is due to the fact that the corresponding bipartite graph also contains elements from $\pz(\sqrt[3]{n},\sqrt{n})$ more than once. For example in case of $n=1202$ and $k=4$ this would cause the constructed matching not to be maximum ($31,37,23,47,11,59$ would be an augmenting path).

On the other hand elements from $\pz(\sqrt[k]{n},\sqrt[3]{n})$ are still usable, since their number of multiples would be smaller than $k$. This second problem can be dealt with using an idea from \cite{dewiljes:2017b}. Repeatedly applying the exchange operation $\mathcal{E}_{B,C}(A):=(A\setminus B)\cup C$, where $A,B,C\subseteq[n]$, and
\begin{lemmaDW1}[de Wiljes \cite{dewiljes:2017b}]
Let $k,n\in\N$ and $A,B,C\subseteq[n]$ with $o_p(C)\leq k-o_p(A\setminus B)$ for all $p\in\pz$. Then $\mathcal{E}_{B,C}(A)\in CS_k(n)$.
\end{lemmaDW1}
as well as
\begin{lemmaDW2}[de Wiljes \cite{dewiljes:2017b}]
Let $k\in\N_{\geq 4}$, $l\in\N$ and $n\in\N$ with $(k-1)\pi(\sqrt[3]{n^2})-(2k-1)\pi(\sqrt{n})>kl-k$. Then for every $A\in CS_k(n)$ there are $l$ primes $q_1,q_2,\dots,q_l\in\pz(\sqrt{n},\sqrt[3]{n^2})$ with $o_{q_i}(A)\leq k-1$ for every $i\in[l]$.
\end{lemmaDW2}
from that paper yields:

\begin{theorem}\label{th:constructionForLargeK}
Let $k\in\N_{\geq 4}$ and $n\in\N$ with
\[
(k-1)\pi(\sqrt[3]{n^2})-(2k-1)\pi(\sqrt{n})>k^2-k.
\]
Then there exists some $A\in CS_k^{max}(n)$ of the form
\[
A=\{1\}\cup\pz(\sqrt[3]{n},n)\cup\{p^2:p\in\pz(\sqrt[3]{n},\sqrt{n})\}\cup X\cup Y\cup Z,
\]
where
\begin{itemize}
\item $o_p(X)=k$ holds for all $p\in\pz(1,\sqrt[3]{n})$,
\item every element from $X$ is of the form $pq$ with $p\in\pz(1,\sqrt[3]{n})$ and $q\in\pz(\sqrt{n},n)$,
\item every element from $Y$ is of the form $pq$ with $p\in\pz(\sqrt[3]{n},\sqrt{n})$ and $q\in\pz(\sqrt{n},n)$,
\item every element from $Z$ is a product of two distinct elements from $p\in\pz(\sqrt[3]{n},\sqrt{n})$.
\end{itemize}
\end{theorem}

\begin{proof}
We start with an arbitrary maximum $(n,k)$-shelf $A$. Then we apply the following steps, where we only move to the next step if the current one cannot be applied anymore.
\begin{itemize}
\item If there exist multiples of elements $p$ from $\pz(1,\sqrt[3]{n})$, which are not of the desired form $pq$ with $q\in\pz(\sqrt{n},n)$, they are exchanged using Lemma DWa and Lemma DWb.
\item If there is some $p\in\pz(\sqrt[3]{n},\sqrt{n})$, for which $|\{p,p^2\}\cap A|<2$ holds, we exchange one of its multiples by one of its first two powers using Lemma DWa.
\item If there exists some $p\in\pz(\sqrt{n},n)$, which is not contained in $A$, we replace a suitable element of $A$ by $p$, again by using Lemma DWa. Note that it can never happen that some $pq$ with $q\in\pz(1,\sqrt[3]{n})$ can be exchanged by $p$ (otherwise $A$ would not have maximal cardinality).
\end{itemize}
The resulting set has the desired form.
\end{proof}

We only have to maximize $|Y|+|Z|$ to get a maximum $(n,k)$-shelf, since the elements of $\pz(1,\sqrt[3]{n})$ can be paired with ``free" primes from $p\in\pz(\sqrt{n},n)$ using Lemma DWb without changing given $Y$ and $Z$.

A promising attempt would be to use the max-flow-min-cut theorem for networks (see for example \cite{bondy:2008}). The following graph $G_k(n)=(V,E)$ for $k\in\N_{\geq 4}$ should be considered:
\begin{itemize}
\item $V=\{s\}\cup\pz(\sqrt[3]{n},\sqrt{n})\cup\pz(\sqrt{n},n)\cup\{t\}$, where $s$ is the \textit{source} and $t$ is the \textit{sink}.
\item The graph has edges between $s$ and every element of $\pz(\sqrt[3]{n},\sqrt{n})$, between $t$ and every element of $\pz(\sqrt{n},n)$, and between $p\in\pz(\sqrt[3]{n},\sqrt{n})$ and $q\in\pz(\sqrt{n},n)$ iff $pq\leq n$.
\item The capacity function $c:E\rightarrow\R$ has values
\[
c(\{s,\cdot\})=k-2,\quad c(\{t,\cdot\})=k-1,\quad\text{and}\quad c(\{p,q\})=1.
\]
\end{itemize}
The given choice of the capacity function ensures that a flow in $G_k(n)$ cannot use elements from $\pz(\sqrt[3]{n},\sqrt{n})$ more than $k-2$ times, elements from $\pz(\sqrt{n},n)$ more than $k-1$ times, and products of the form $pq$ with $p\in\pz(\sqrt[3]{n},\sqrt{n})$ and $q\in\pz(\sqrt{n},n)$ more than once.

We now have to find a maximum flow $f$ in $G_k(n)$ while also guaranteeing that the differences of $f$ and the capacity on the edges starting at $s$ yield a degree sequence of a graph or are at least close to one. Then the edges of the form $\{p,q\}$ used by $f$ (meaning where $f$ has value $1$) produce the set $Y$ in Theorem \ref{th:constructionForLargeK} and from the degree sequence we can construct $Z$ (by the same procedure as in the lower bound in \cite{dewiljes:2017a}). The author does not know yet how to solve the ``close to degree sequence" problem.


\begin{thebibliography}{9}

\bibitem{berge:1957}
C. Berge,
Two theorems in graph theory,
\emph{Proceedings of the National Academy of Sciences of the United States of America} {\bfseries 43} (1957), no. 9, 842--844.

\bibitem{bondy:2008}
A. Bondy and U. S. R. Murty,
Graph {T}heory,
Springer,
London,
2008.

\bibitem{dewiljes:2017a}
J.-H. de Wiljes,
Complete Subgraphs of the Coprime Hypergraph of Integers I: Introduction and Bounds,
\emph{European Journal of Mathematics} {\bfseries 3} (2017), no. 2, 379--386.

\bibitem{dewiljes:2017b}
J.-H. de Wiljes,
Complete Subgraphs of the Coprime Hypergraph of Integers II: Structural Properties,
to appear in \emph{European Journal of Mathematics}.

\end{thebibliography}
\end{document}